\documentclass[11pt,twoside,reqno]{amsart}
\allowdisplaybreaks
\usepackage{amsmath,amstext,amssymb,epsfig,multicol,enumerate}
\usepackage{graphicx}
\usepackage{mathrsfs}
\calclayout
\makeatletter
\renewcommand{\@seccntformat}[1]{\bf\csname the#1\endcsname.}
\renewcommand{\section}{\@startsection{section}{1}
	\z@{.7\linespacing\@plus\linespacing}{.5\linespacing}
	{\normalfont\upshape\bfseries\centering}}
\renewcommand{\@biblabel}[1]{\@ifnotempty{#1}{#1.}}
\makeatother
\theoremstyle{plain}
\newtheorem{thm}{Theorem}[section]

\newtheorem{prop}[thm]{Proposition}
\newtheorem{cor}[thm]{Corollary}

\theoremstyle{definition}

\newtheorem{defn}[thm]{Definition}

\usepackage{cancel}
\usepackage[parfill]{parskip}
\usepackage[german]{varioref}
\usepackage[all]{xy}
\usepackage{color}

\def\A{{\mathcal A}}

\def \>{\succ}
\def \<{\prec}

\def\A{{\mathbb A}}

\begin{document}	
\title[Basdouri Imed\textsuperscript{1}, Jean Lerbet\textsuperscript{2}, Bouzid Mosbahi\textsuperscript{3}]{CENTRAL DERIVATIONS OF LOW-DIMENSIONAL ZINBIEL ALGEBRAS}
	\author{Basdouri Imed\textsuperscript{1}, Jean Lerbet\textsuperscript{2}, Bouzid Mosbahi\textsuperscript{3}}
 \address{\textsuperscript{1}Department of Mathematics, Faculty of Sciences, University of Gafsa, Gafsa, Tunisia}
 \address{\textsuperscript{2}Laboratoire de Mathématiques et Modélisation d’Évry (UMR 8071) Université d’Évry Val d’Essonne I.B.G.B.I., 23 Bd. de France, 91037 Évry Cedex, France}
\address{\textsuperscript{3}Department of Mathematics, Faculty of Sciences, University of Sfax, Sfax, Tunisia}
 \email{\textsuperscript{1}basdourimed@yahoo.fr}
 \email{\textsuperscript{2}jean.lerbet@ufrst.univ-evry.fr}
\email{\textsuperscript{3}mosbahi.bouzid.etud@fss.usf.tn}


	\date{\today}

	\begin{abstract}
The study of central derivations in low-dimensional algebraic structures is a crucial area of research in mathematics, with applications in understanding the internal symmetries and deformations of these structures. In this article, we investigate the central derivations of complex Zinbiel algebras of dimension $\leq 4$. Key properties of the central derivation algebras are presented, including their structures and dimensions. The results are summarized in a tabular format, providing a clear classification of decomposable and indecomposable centroids based on these derivations. Specifically, we show that the centroid of two-dimensional Zinbiel algebras is indecomposable, while in three-dimensional Zinbiel algebras, centroids such as $\A_3^3$, $\A_4^3$, $\A_6^3$, and $\A_7^3$ are decomposable. For four-dimensional Zinbiel algebras, centroids including $\A_1^4$, $\A_3^4$, $\A_5^4$, $\A_9^4$, $\A_{10}^4$, $\A_{11}^4$, and $\A_{16}^4$ are decomposable. Furthermore, the dimensions of central derivation algebras vary across different dimensions: two-dimensional Zinbiel algebras have central derivation dimensions of one, while in three-dimensional and four-dimensional cases, these dimensions range from zero to nine.
\end{abstract}

\maketitle \section{ Introduction}\label{introduction}
Non-associative algebras have become a central object of study in modern algebraic theory, particularly through their connections to Lie algebras and their various generalizations, such as Malcev algebras, Lie superalgebras, binary Lie algebras, Leibniz algebras, and others. Among these, Leibniz algebras have gained significant attention for their non-associative yet structured properties, defined by the identity
\begin{align*}
[a, [b, c]] &= [[a, b], c] - [[a, c], b]
\end{align*}
The study of these algebras has led to the exploration of their dual structures, notably Zinbiel algebras, which are considered the Koszul duals of Leibniz algebras in the framework of J.L. Loday's theory \cite{1}. In \cite{2}, Zinbiel algebras were introduced as dual Leibniz algebras, with further details provided in works such as \cite{2,3}.

A Zinbiel algebra consists of a vector space  $\A$ with a bilinear map $\bullet: \A \times \A \to \A$ satisfying the identity
\begin{align*}
(a \bullet b) \bullet c &= a \bullet (b \bullet c) + a \bullet (c \bullet b),
\end{align*}
This structure allows for the formulation of commutative associative algebraic structures through the new product  $a \bullet b + b \bullet a$, making Zinbiel algebras closely related to dendriform algebras and pre-Lie algebras see \cite{4,5}. In fact, dendriform algebras, defined by two bilinear operations $\prec$  and $\succ$  on a vector space $\A$, yield associative structures when combined appropriately. Specifically, a Zinbiel algebra can be viewed as a particular case of a dendriform algebra with a specific relationship between the operations $\prec$ and  $\succ$. Pre-Lie algebras, which share a similar structure, relate to Lie algebras through a commutative product. These algebraic objects can be visualized through the following diagram, reflecting their interconnections:

\begin{center}
    $\xymatrix{
        & Zinbiel \ar[r] \ar[d] & Dendriform \ar[r] \ar[d] & Pre-Lie \ar[d] \\
        & Commutative \ar[r] & Associative \ar[r] & Lie
    }$
\end{center}

Furthermore, Zinbiel algebras hold important relationships with pre-Poisson algebras and other structures, showing their versatility and importance in the wider context of algebra theory. Although significant progress has been made in understanding Zinbiel algebras, particularly their connections to dendriform and pre-Lie algebras, many challenges remain due to their inherent non-associativity.

This paper focuses on central derivations of low-dimensional Zinbiel algebras, specifically those of dimensions two, three, and four. Building on existing classifications of complex Zinbiel algebras see \cite{6,7,8,9,10}, we aim to systematically compute central derivations, classify centroids for decomposability, and provide illustrative examples. To facilitate our computations, we employ symbolic algebra software, including Maple, ensuring rigorous and efficient analysis.
The paper is organized as follows: Section 1 provides an overview of Zinbiel algebras and summarizes key results. Section 2 introduces the definitions and preliminary concepts needed for the study. Section 3 presents an algorithm for finding central derivations in Zinbiel algebras. Finally, using classifications of low-dimensional complex Zinbiel algebras from earlier work, we compute central derivations, examine centroids for decomposability, and provide examples generated using Maple software.

\section{ Prelimieries}
\begin{defn}
An algebra  $\A$ over a field $\mathbb{K}$ is a vector space over $\mathbb{K}$ equipped with a bilinear map
\begin{align*}
\pi: \A \times \A \to \A,
\end{align*}
such that
\begin{align*}
\pi(\alpha a + \beta b, c) &= \alpha \pi(a, c) + \beta \pi(b, c),\\
\pi(c, \alpha a + \beta b) &= \alpha \pi(c, a) + \beta \pi(c, b),
\end{align*}
where  $a, b, c \in \A$  and  $\alpha, \beta \in \mathbb{K}$.
\end{defn}

\begin{defn}
Let $\A$ be a vector space over $\mathbb{K}$ with a bilinear product $\cdot : \A \times \A \to \A$. Then $(\A, \cdot)$ is called a \textit{pre-Lie algebra} if for any $a, b, c \in \A$,
\begin{align*}
(a \cdot b) \cdot c - a \cdot (b \cdot c) &= (b \cdot a) \cdot c - b \cdot (a \cdot c).
\end{align*}
\end{defn}

\begin{defn}
A dendriform algebra $\A$  over a field $\mathbb{K}$  is a  $\mathbb{K}$-vector space equipped with two binary operations
\begin{align*}
\succ: \A \otimes \A \rightarrow \A, \quad \prec: \A \otimes \A \rightarrow \A,
\end{align*}
satisfying the following identities for all $a, b, c \in \A$:
\begin{align*}
    (a \prec b) \prec c &= a \prec (b \prec c) + a \prec (b \succ c),\\
    (a \succ b) \prec c &= a \succ (b \prec c), \\
    a \succ (b \succ c)&=(a \prec b) \succ c + (a \succ b) \succ c.
\end{align*}
\end{defn}

\begin{defn}
A Zinbiel algebra $\A$ over a field $\mathbb{K}$ is an algebra equipped with a bilinear operation  $\bullet : \A \times \A \to \A$ satisfying the Zinbiel identity:
\begin{align*}
(a \bullet b) \bullet c &= a \bullet (b \bullet c) + a \bullet (c \bullet b), \quad \forall \; a, b, c \in \A.
\end{align*}
\end{defn}

\begin{defn}
A derivation of a Zinbiel algebra  $\A$  is a linear map  $D: \A \to \A$  satisfying:
\begin{align*}
D(a \bullet b) &= D(a) \bullet b + a \bullet D(b), \quad \forall \; a, b \in \A.
\end{align*}
The set of all derivations of $\A$ is denoted by $Der(\A)$.
\end{defn}

\begin{defn}
For a Zinbiel algebra  $\A$, define a descending sequence of subspaces as:
\begin{align*}
\A^1 &= \A, \quad \A^{k+1} = \A^k \bullet \A, \quad k \geq 1.
\end{align*}
It follows that:
\begin{align*}
\A^1 \supseteq \A^2 \supseteq \A^3 \supseteq \cdots.
\end{align*}
\end{defn}

\begin{defn}
Let $\A$ and $\A_1$ be two Zinbiel algebras over a field $\mathbb{K}$. A mapping $\psi : \A \to \A_1$ is called a \textit{homomorphism} if it satisfies
\begin{align*}
\psi(a \bullet b) &= \psi(a) \bullet \psi(b), \quad \text{for all } a, b \in \A.
\end{align*}
We call $\psi$ an \textit{isomorphism} if it is bijective and an \textit{endomorphism} if it is a linear map $\psi : \A \to \A$. The set of all endomorphisms of $\A$ is denoted by $End(\A)$.
\end{defn}

\begin{defn}
Let $\A$ be a Zinbiel algebra over a field $\mathbb{K}$. The set
\begin{align*}
\Gamma(\A) &= \{ \phi \in End(\A) \mid \phi(a \bullet b) = \phi(a) \bullet b = a \bullet(b), \quad \text{for all } a, b \in \A \}
\end{align*}
is called the \textit{centroid} of $\A$.
\end{defn}

\begin{defn}
Let  $\A_1$ be a nonempty subset of a Zinbiel algebra  $\A$. The subset
\begin{align*}
C_\A(\A_1) &= \{ a \in \A \mid a \bullet \A_1 = \A_1 \bullet a = 0 \}
\end{align*}
is called the \textit{centralizer} of $\A_1$ in  $\A$. In particular, $C_\A(\A) = C(\A)$ is the \textit{center} of $\A$. Additionally, an ideal $I$ of a Zinbiel algebra $\A$ satisfies  $\A \bullet I \subseteq I$  and  $I \bullet \A \subseteq I$.
\end{defn}

\begin{defn}
A Zinbiel algebra $\A$ is called \textit{indecomposable} if it cannot be written as a direct sum of its ideals. Otherwise, $\A$ is called \textit{decomposable}.
\end{defn}

\begin{defn}
Let $\A$ be a Zinbiel algebra and $\phi \in End(\A)$. The map $\phi$ is called a \textit{central derivation} if  $\phi(\A) \subseteq C(\A)$ and  $\phi(\A \bullet \A) = 0$.
\end{defn}

The set of all central derivations of a Zinbiel algebra $\A$, denoted by $C_D(\A)$, forms an associative algebra with respect to the composition operation  $\circ$, and it is a Lie algebra with respect to the bracket defined by
\begin{align*}
[a_1, a_2] &= a_1 \circ a_2 - a_2 \circ a_1, \quad \text{for all } a_1, a_2 \in C_D(\A).
\end{align*}

\begin{thm}
Let $\xi : \A_1 \to \A_2$  be an isomorphism of Zinbiel algebras  $(\A_1, \star)$ and  $(\A_2, \ast)$ over a field $\mathbb{K}$. The mapping $\phi : End(\A_1) \to End(\A_2)$ defined by $\phi(c) = \xi \circ c \circ \xi^{-1}$ is an isomorphism of $CD(\A_1)$ and $CD(\A_2)$, that is,
\begin{align*}
\phi(CD(\A_1)) &= CD(\A_2).
\end{align*}
\end{thm}

\begin{proof}
Due to the isomorphism relation, we have
\begin{align*}
a \ast b &= \xi(\xi^{-1}(a) \star \xi^{-1}(b)).
\end{align*}
Assume  $c \in End(\A_1$)  such that
\begin{align*}
c(\xi^{-1}(a_1)) \star \xi^{-1}(a_2) + \xi^{-1}(a_1) \star c(\xi^{-1}(a_2)) &= c(\xi^{-1}(a_1) \star \xi^{-1}(a_2)).
\end{align*}
Applying the mapping $\xi$ on this equation, we have
\begin{align*}
\xi \circ c \circ \xi^{-1}(a_1 \ast a_2) &= \xi \circ c \circ \xi^{-1}(a_1) \ast a_2 + a_1 \ast \xi \circ c \circ \xi^{-1}(a_2),
\end{align*}
that is,  $\xi \circ c \circ \xi^{-1} \in CD(\A_2)$. Thus,
\begin{align*}
\phi(\mathrm{CD}(\A_1)) &= \mathrm{CD}(\A_2).
\end{align*}
\end{proof}

\begin{prop}
Let $\A$ be a Zinbiel algebra. Then
\begin{itemize}
    \item[(i)] $CD(\A) \subseteq \mathrm{Der}(\A)$;
    \item[(ii)] $[ \Gamma(\A), \Gamma(\A) ] \subseteq CD(\A)$;
    \item[(iii)] $(\phi \circ D)(a \bullet b) = (\phi \circ D)(a) \bullet b + a \bullet (\phi \circ D)(b)$;
    \item[(iv)] $[D, \phi](a \bullet b) = [D, \phi](a) \bullet b + a \bullet [d, \phi](b)$.
\end{itemize}
\end{prop}

\begin{proof}
We prove $(i)$, and the others can be obtained by definitions.
Let $\A$ be a Zinbiel algebra, $\phi_1 \in CD(\A)$, and for all $a, b \in \A$. We need to show that
\begin{align*}
\phi_1(a \bullet b) &= \phi_1(a) \bullet b = a \bullet \phi_1(b) = 0.
\end{align*}
Thus,
\begin{align}
\phi_1(a \bullet b) &= \phi_1(a) \bullet b = 0 \label{ed1}\\
\phi_1(a \bullet b) &= a \bullet \phi_1(b) = 0.\label{eq2}
\end{align}
If we add (\ref{ed1}) with (\ref{eq2}), we get
\begin{align*}
\phi_1(a \bullet b) + \phi_1(a \bullet b) &= \phi_1(a) \bullet b + a \bullet \phi_1(b).
\end{align*}
Therefore,
\begin{align*}
\phi_1(a \bullet b) &= \phi_1(a) \bullet b + a \bullet \phi_1(b) \subseteq Der(\A),
\end{align*}
since $\phi_1(a \bullet b) = 0$.
\end{proof}

\begin{thm}\label{t2}
Let $\A$ be a Zinbiel algebra. Then for any  $\phi \in \Gamma(\A)$  and  $d \in Der(\A)$, the following statements hold:
\begin{enumerate}
    \item $Der(\A) \cap \Gamma(\A) = CD(\A)$,
    \item  $d \circ \phi \in \Gamma(\A)$ if and only if  $\phi \circ d$  is a central derivation of  $\A$,
    \item \( d \circ \phi \in Der(\A) \) if and only if  $[d, \phi]$ is a central derivation of  $\A$.
\end{enumerate}
\end{thm}

\begin{proof}
\begin{enumerate}
    \item Suppose $\phi \in Der(\A) \cap \Gamma(\A)$. Then $\phi \in Der(\A)$ and $\phi \in \Gamma(\A)$. Let $a, b \in \A$. From the definition of a derivation, we have:
\begin{align*}
\phi(a \bullet b) &= \phi(a) \bullet b + a \bullet \phi(b).
\end{align*}
From the centroid property, we have:
\begin{align*}
\phi(a \bullet b) &= \phi(a) \bullet b = a \bullet \phi(b).
\end{align*}
Thus,
\begin{align*}
\phi(a \bullet b) &= \phi(a) \bullet b = a \bullet \phi(b) = 0.
\end{align*}
Therefore, $\phi(\A^2) = 0$ and  $\phi(\A) \subseteq C(\A)$. Hence, $\Gamma(\A) \cap Der(\A) \subseteq CD(\A)$.

To prove the reverse inclusion, suppose  $\phi \in C(\A)$. Then:
\begin{align*}
\phi(a \bullet b) &= 0 = \phi(a) \bullet b = a \bullet \phi(b).
\end{align*}
This implies $\phi \in \Gamma(\A) \cap Der(\A)$. Therefore,
\begin{align*}
Der(\A) \cap \Gamma(\A) &= C(\A).
\end{align*}

\item Assume $\phi \in \Gamma(\A)$ and  $D \in Der(\A)$. Let  $a, b \in \A$. If $D \circ \phi \in \Gamma(\A)$, then:

\begin{align*}
(D \circ \phi)(a \bullet b)&= ((D \circ \phi)(a)) \bullet b = a \bullet ((D \circ \phi)(b)).
\end{align*}
On the other hand:
\begin{align*}
(D \circ \phi)(a \bullet b) &= D(\phi(a) \bullet b) = D(a \bullet \phi(b)).
\end{align*}
Using the derivation property,
\begin{align*}
D(\phi(a) \bullet b) &= (D \circ \phi)(a) \bullet b + \phi(a) \bullet D(b),
\end{align*}
and
\begin{align*}
D(a \bullet \phi(b)) &= a \bullet (D \circ \phi)(b) + D(a) \bullet \phi(b).
\end{align*}
Equating these expressions, we find:
\begin{align*}
(D \circ \phi)(a \bullet b) &= (D \circ \phi)(a) \bullet b + a \bullet (D \circ \phi)(b).
\end{align*}
Thus,  $D \circ \phi \in Der(\A)$.

To show the reverse inclusion, suppose $D \circ \phi \in Der(\A)$. Then:
\begin{align*}
(D \circ \phi)(a \bullet b) &= (D \circ \phi)(a) \bullet b + a \bullet (D \circ \phi)(b).
\end{align*}
By substituting the definitions and rearranging, it follows that  $D \circ \phi \in \Gamma(\A)$.

\item Assume  $[D, \phi] \in Der(\A)$. For any $a, b \in \A$,
\begin{align*}
[D, \phi](a \bullet b) &= D(\phi(a \bullet b)) - \phi(D(a \bullet b)).
\end{align*}
Using the derivation properties of $D$ and $\phi$, and simplifying, we find that:
\begin{align*}
[a, [D, \phi](b)] &= 0.
\end{align*}
By the properties of Zinbiel algebras, this implies  $[D, \phi] \in C(\A)$.

Conversely, if $[D, \phi] \in C(\A)$, then:
\begin{align*}
[a, [D, \phi](b)] &= 0.
\end{align*}
This ensures $D \circ \phi \in Der(\A)$, completing the proof.
\end{enumerate}
\end{proof}

\begin{thm}\label{t1}
Let $\A = \A_1 \oplus \A_2$ , where $\A_1$ and  $\A_2$ are ideals of  $\A$. Then,
\begin{align*}
\Gamma(\A) &= \Gamma(\A_1) \oplus \Gamma(\A_2) \oplus C_1 \oplus C_2,
\end{align*}
where  $C_i$ is defined as
\begin{align*}
C_i &= \left\{ \phi \in End(\A_i, \A_j) \mid \phi(\A_i) \subseteq C(\A_j), \phi \A_i = 0 \right\}, \quad 1 \leq i \neq j \leq 2.
\end{align*}
\end{thm}

\begin{proof}
Let $\rho_i: \A \to \A_i$ be the canonical projection,  $i = 1, 2$. Then  $\rho_1, \rho_2 \in \Gamma(\A)$  and $\rho_1 + \rho_2 = id_\A$.

For any  $\phi \in \Gamma(\A)$, we have
\begin{align*}
\phi &= (\rho_1 + \rho_2)\phi(\rho_1 + \rho_2) = \rho_1\phi\rho_1 + \rho_1\phi\rho_2 + \rho_2\phi\rho_1 + \rho_2\phi\rho_2.
\end{align*}
Thus,  $\rho_i\phi\rho_j \in \Gamma(\A),  i, j = 1, 2$, and  $\rho_1\Gamma(\A)\rho_1 + \rho_1\Gamma(\A)\rho_2 + \rho_2\Gamma(\A)\rho_1 + \rho_2\Gamma(\A)\rho_2$  is a direct sum.

Now, we define a mapping  $\rho_1\Gamma(\A)\rho_1 \to \Gamma(\A_1)$ such that $ \rho_1 f \rho_1 \mapsto \rho_1 f \rho_1|_{\A_1}$ , for any $f \in \Gamma(\A)$.

If $\rho_1 f \rho_1|_{\A_1} = 0$, then $\rho_1 f \rho_1|_{\A_2} = 0$, which implies  $\rho_1 f \rho_1 = 0$ .Hence, the above mapping is injective.

For any  $\phi \in \Gamma(A_1)$, we can extend  $\phi$ to $\A_2$  such that $\phi|_{\A_2} = 0$. Then the extended  $\phi$ is in  $\Gamma(\A)$ and  $\rho_1 \phi \rho_1|_{\A_1} = \phi$. Thus,  $\rho_1 \Gamma(\A) \rho_1 \cong \Gamma(\A_1)$ as a vector space. Similarly, we obtain  $\rho_2 \Gamma(\A) \rho_2 \cong \Gamma(\A_2)$.

Next, define a mapping  $\rho_1 \Gamma(\A) \rho_2 \to C_2$ such that  $\rho_1 f \rho_2 \mapsto \rho_1 f \rho_2|_{L_2}$.

For $a \in \A_2, b \in \A_1$,
\begin{align*}
\rho_1 f \rho_2(a)b &= \rho_1 f \rho_2(ab) = 0,
\end{align*}
and
\begin{align*}
b(\rho_1 f \rho_2(a)) &= \rho_1 f \rho_2(ba) = 0.
\end{align*}

Thus, $\rho_1 f \rho_2(a) \in C(\A_1)$. For  $a, b \in \A_2, \rho_1 f \rho_2(ab) = 0$ , so $\rho_1 f \rho_2(\A_1) = 0$  and  $\rho_1 f \rho_2|_{\A_2} \in C_2$.

If  $\rho_1 f \rho_2|_{\A_2} = 0$ and  $\rho_1 f \rho_2|_{\A_1} = 0$, then $\rho_1 f \rho_2|_\A = 0$, and the above mapping is injective.

For any  $\phi \in C_2$, extend  $\phi$ to  $\A_1$ and denote it by  $\bar{\phi}$, such that  $\bar{\phi}|_{\A_1} = 0$. Then  $\bar{\phi} \in \Gamma(\A)$ by the following conditions:
\begin{align*}
\bar{\phi}(ab) &= \bar{\phi}(a_1 b_1 + a_2 b_2) = \phi(a_2 b_2) = 0,\\
a \bar{\phi}(b) &= (a_1 + b_2) \bar{\phi}(b_1 + b_2) = a_2 \phi(b_2) = 0,\\
\bar{\phi}(a)b &= \phi(a_2)(b_1 + b_2) = \phi(a_2)b = 0.
\end{align*}

Therefore,  $\rho_1 \bar{\phi} \rho_2(a_2) = \rho_1 \bar{\phi}(a_2) = \phi(a_2)$, and $\rho_1 \phi \rho_2 \mapsto \phi$, which shows that the mapping above is onto.

Thus,  $\rho_1 \Gamma(\A) \rho_2 \cong C_2$ as a vector space. Similarly,  $\rho_2 \Gamma(\A) \rho_1 \cong C_1$.
The theorem is proven.
\end{proof}

\begin{cor}
If  $CD(\A) = 0$, then  $\Gamma(\A)$ is decomposable.
\end{cor}

\begin{proof}
Let $\A = \A_1 \oplus \A_2$, where $\A_1$  and  $\A_2$  are ideals of  $\A$. From Theorem \ref{t1}, the centroid  $\Gamma(\A)$ is given by:
\begin{align*}
\Gamma(\A) &= \Gamma(\A_1) \oplus \Gamma(\A_2) \oplus C_1 \oplus C_2,
\end{align*}
where  $C_i = \{ \phi \in End(\A_i, \A_j) \mid \phi(\A_i) \subseteq C(\A_j), \, \phi \A_i = 0 \}, \, 1 \leq i \neq j \leq 2$.

Suppose $ CD(\A) = 0$, which implies that the center of  $\A$, denoted by $C(\A)$, is trivial. That is,$C(\A) = \{ 0 \}$. Consequently:
\begin{align*}
C(\A_j) &= \{ 0 \} \quad \text{for } j = 1, 2.
\end{align*}

Now, consider the components  $C_1$  and  $C_2$. For  $\phi \in C_1$ , by definition:
\begin{align*}
\phi(\A_1) \subseteq C(\A_2), \quad \phi \A_1 &= 0.
\end{align*}
Since  $C(\A_2) = \{ 0 \}$, it follows that  $\phi(\A_1) = 0$, which implies  $\phi = 0$. Thus, $ C_1 = \{ 0 \}$.

Similarly, for $\phi \in C_2$, we have:
\begin{align*}
\phi(\A_2) \subseteq C(\A_1), \quad \phi \A_2 &= 0.
\end{align*}
Since  $C(\A_1) = \{ 0 \}$, it follows that $\phi(\A_2) = 0$, which implies  $\phi = 0$. Thus,  $C_2 = \{ 0 \}$.

Substituting  $C_1 = C_2 = \{ 0 \}$ into the expression for  $\Gamma(\A)$, we obtain:
\begin{align*}
\Gamma(\A) &= \Gamma(\A_1) \oplus \Gamma(\A_2).
\end{align*}

Hence, $\Gamma(\A)$ is a direct sum of  $\Gamma(\A_1)$ and  $\Gamma(\A_2)$, and therefore  $\Gamma(\A)$ is decomposable.
\end{proof}

\section{ An algorithm for finding central derivations}
Firstly, let  $\{e_1, e_2, \dots, e_n\}$ be a basis of an  $n$-dimensional Zinbiel algebra $\A$. Then
\begin{align*}
e_i \bullet e_j &= \sum_{k=1}^n \gamma_{ij}^k e_k, \quad i, j = 1, 2, \dots, n,
\end{align*}
where the coefficients \( \gamma_{ij}^k \) of the above linear combinations are called the structure constants.

An element $a$  of central derivation, $CD(\A)$, being a linear transformation of the Zinbiel algebra $\A$, is represented in a square matrix form \( [a_{ij}]_{i,j=1,2,\dots,n} \), that is,
\begin{align*}
\phi(e_i) &= \sum_{t=1}^n a_{it} e_t, \quad i = 1, 2, \dots, n.
\end{align*}

According to Theorem \ref{t2}, the central derivation $CD(\A)$ is the intersection between the centroid  $\Gamma(\A)$ and the derivation $Der(\A)$. From the definition of centroid and central derivation, we derive an algorithm to find the central derivation $CD(\A)$, as follows:
\begin{align*}
\sum_{t=1}^n \gamma_{ij}^t a_{tk} &= \sum_{t=1}^n a_{it} \gamma_{tj}^k = \sum_{t=1}^n a_{jt} \gamma_{it}^k = 0, \quad \forall i, j, k = 1, 2, \dots, n.
\end{align*}

This approach can be applied to find the central derivations of complex Zinbiel algebras in dimensions 2, 3, and 4. Additionally, using the classification results from \([5]\) and \([4]\) together with the algorithm mentioned above, we compute the central derivations, and the results are summarized in tabular form.

\begin{thm}
Any $2$-dimensional Zinbiel algebra $\A$ isomorphic to
one of the following nonisomorphic Zinbiel algebras $\A_2^{1}:\;e_1\bullet e_1=e_2$.
\end{thm}
\begin{cor}
The centroid of two dimensional Zinbiel algebras is indecomposable.
\end{cor}
\[
\text{Table 1: central derivation of two-dimensional complex Zinbiel algebras}
\]
\[
\begin{array}{|c|c|c|}
\hline
\hline
\textbf{Isomorphism Class} & \textbf{Central Derivation} & \textbf{Dimension} \\
\hline
\A_2^{1} &
\left(\begin{array}{cc}
0 & 0 \\
a_{21} & 0
\end{array}\right) & 1 \\
\hline
\end{array}
\]
\begin{thm}
Any $3$-dimensional Zinbiel algebra $\A$ isomorphic to one of following non-isomorphic Zinbiel algebras\\
$\A_3^{1}:\;e_i\bullet e_j=0$\\
$\A_3^{2}:\;e_1\bullet e_1=e_3$\\
$\A_3^{3}:\;e_1\bullet e_1=e_3, e_2\bullet e_2=e_3$\\
$\A_3^{4}:\;e_1\bullet e_2=\frac{1}{2}e_3, e_2\bullet e_1=\frac{-1}{2}e_3$\\
$\A_3^{5}:\;e_2\bullet e_1=e_3$\\
$\A_3^{6}:\;e_1\bullet e_1=e_3,e_1\bullet e_2=e_3,e_2\bullet e_2=\lambda e_3, \quad \lambda  \neq 0 $\\
$\A_3^{7}:\;e_1\bullet e_1=e_2,e_1\bullet e_2=\frac{1}{2}e_3,e_2\bullet e_1=e_3$
\end{thm}
\begin{cor}
In any three-dimensional Zinbiel algebras, the centroids of $\A_3^3, \A_3^4, \A_3^6$ and $\A_3^7$ are decomposable.
\end{cor}
\[
\text{Table 2: central derivation of three-dimensional complex Zinbiel algebras}
\]
\[
\begin{array}{|c|c|c|}
\hline
\hline
\textbf{Isomorphism Class} & \textbf{Central Derivation} & \textbf{Dimension} \\
\hline
\A_3^{1}
&\left(\begin{array}{ccc}
a_{11}&a_{12}&a_{13}\\
a_{21}&a_{22}&a_{23}\\
a_{31}&a_{32}&a_{33}
\end{array}\right) & 9 \\
\hline
\A_3^{2}
&\left(\begin{array}{ccc}
0&0&0\\
a_{21}&a_{22}&a_{23}\\
a_{31}&0&0
\end{array}\right) & 4 \\
\hline
\A_3^{3}
&\left(\begin{array}{ccc}
0&0&0\\
0&0&0\\
0&0&0
\end{array}\right) & 0 \\
\hline
\A_3^{4}
&\left(\begin{array}{ccc}
0&0&0\\
0&0&0\\
0&0&0
\end{array}\right) & 0 \\
\hline
\A_3^{5}
&\left(\begin{array}{ccc}
0&0&0\\
a_{21}&a_{22}&a_{23}\\
0&a_{32}&0
\end{array}\right) & 4 \\
\hline
\A_3^{6}
&\left(\begin{array}{ccc}
0&0&0\\
0&0&0\\
0&0&0
\end{array}\right) & 0 \\
\hline
\A_3^{7}
&\left(\begin{array}{ccc}
0&0&0\\
0&0&0\\
0&0&0
\end{array}\right) & 0 \\
\hline
\end{array}
\]
\begin{thm}
Any $4$-dimensional Zinbiel algebra $\A$ isomorphic to one of following non-isomorphic Zinbiel algebras\\
$\A_4^1:\;e_1\bullet e_1=e_2, e_1\bullet e_2=e_3,e_2\bullet e_1=2e_3,e_1\bullet e_3=e_4, e_2\bullet e_2=3e_4, e_3\bullet e_1=3e_4 $\\
$\A_4^{2}:\;e_1\bullet e_1=e_3, e_1\bullet e_2=e_4,e_1\bullet e_3=e_4,e_3\bullet e_1=2e_4$\\
$\A_4^3:\;e_1\bullet e_1=e_3, e_1\bullet e_3=e_4,e_2\bullet e_2=e_4,e_3\bullet e_1=2e_4$\\
$\A_4^{4}:\;e_1\bullet e_2=e_3, e_1\bullet e_3=e_4,e_2\bullet e_1=-e_3$\\
$\A_4^5:\;e_1\bullet e_2=e_3, e_1\bullet e_3=e_4,e_2\bullet e_1=-e_3,e_2\bullet e_2=e_4$\\
$\A_4^{6}:\;e_1\bullet e_1=e_4, e_1\bullet e_2=e_3,e_2\bullet e_1=-e_3,e_2\bullet e_2=-2e_3+e_4$\\
$\A_4^7:\;e_1\bullet e_2=e_3, e_2\bullet e_1=e_4,e_2\bullet e_2=-e_3$\\
$\A_4^{8}:\;e_1\bullet e_1=e_3, e_1\bullet e_2=e_4,e_2\bullet e_1=-\alpha e_3,e_2\bullet e_2=-e_4$\\
$\A_4^9:\;e_1\bullet e_1=e_4, e_1\bullet e_2=\alpha e_4,e_2\bullet e_1=-\alpha e_4,e_2\bullet e_2=e_4,e_3\bullet e_3=e_4$\\
$\A_4^{10}:\;e_1\bullet e_1=e_4, e_1\bullet e_3=e_4,e_2\bullet e_1=-e_4,e_2\bullet e_2=e_4,e_3\bullet e_1=e_4$\\
$\A_4^{11}:\;e_1\bullet e_1=e_4, e_1\bullet e_2=e_4,e_2\bullet e_1=-e_4,e_3\bullet e_3=e_4$\\
$\A_4^{12}:\;e_1\bullet e_2=e_3, e_2\bullet e_1=e_4$\\
$\A_4^{13}:\;e_1\bullet e_2=e_3, e_2e_1=e_4$\\
$\A_4^{14}:\;e_1\bullet e_2=e_3, e_2\bullet e_1=e_4$\\
$\A_4^{15}:\;e_1\bullet e_2=e_3, e_2\bullet e_1=e_4$\\
$\A_4^{16}:\;e_1\bullet e_2=e_3, e_2\bullet e_1=e_4$
\end{thm}
\begin{cor}
In any four-dimensional Zinbiel algebras, the centroids of $\A_4^1, \A_4^3, \A_4^5,\A_4^9, \A_4^{10}, \A_4^{11}$ and $\A_4^{16}$ are decomposable.
\end{cor}
\[
\text{Table 3: central derivation of four-dimensional complex Zinbiel algebras}
\]
\[
\begin{array}{|c|c|c|}
\hline
\hline
\textbf{Isomorphism Class} & \textbf{Central Derivation} & \textbf{Dimension} \\
\hline
\A_4^{1}
&\left(\begin{array}{cccc}
0&0&0&0\\
0&0&0&0\\
0&0&0&0\\
0&0&0&0
\end{array}\right) & 0 \\
\hline
\A_4^{2}
&\left(\begin{array}{cccc}
0&0&0&0\\
-a_{31}&0&-2a_{41}&0\\
a_{31}&0&2a_{41}&0\\
a_{41}&0&0&0
\end{array}\right) & 2 \\
\hline
\A_4^{3}
&\left(\begin{array}{cccc}
0&0&0&0\\
0&0&0&0\\
0&0&0&0\\
0&0&0&0
\end{array}\right) & 0 \\
\hline
\A_4^{4}
&\left(\begin{array}{cccc}
0&0&0&0\\
0&0&0&0\\
0&0&0&0\\
a_{41}&0&0&0
\end{array}\right) & 1 \\
\hline
\A_4^{5}
&\left(\begin{array}{cccc}
0&0&0&0\\
0&0&0&0\\
0&0&0&0\\
0&0&0&0
\end{array}\right) & 0 \\
\hline
\A_4^{6}
&\left(\begin{array}{cccc}
0&0&0&0\\
0&0&0&0\\
-a_{42}&2a_{42}+a_{41}&0&0\\
a_{41}&a_{42}&0&0
\end{array}\right) & 2 \\
\hline
\A_4^{7}
&\left(\begin{array}{cccc}
0&0&0&0\\
0&0&0&0\\
-a_{32}&a_{32}&0&0\\
0&a_{42}&0&0
\end{array}\right) & 2 \\
\hline
\end{array}
\]
\[
\begin{array}{|c|c|c|}
\hline
\hline
\textbf{Isomorphism Class} & \textbf{Central Derivation} & \textbf{Dimension} \\
\hline
\A_4^{8}
&\left(\begin{array}{cccc}
0&0&0&0\\
0&0&0&0\\
a_{31}&\alpha a_{31}&0&0\\
-a_{42}&a_{42}&0&0
\end{array}\right)\; (\alpha \neq 0)
& 2 \\
&\left(\begin{array}{cccc}
0&0&0&0\\
0&0&0&0\\
a_{31}&0&0&0\\
-a_{42}&a_{42}&0&0
\end{array}\right)\; (\alpha = 0)
& 2 \\
\hline
\A_4^{9}
&\left(\begin{array}{cccc}
0&0&0&0\\
0&0&0&0\\
0&0&0&0\\
0&0&0&0
\end{array}\right) & 0 \\
\hline
\A_4^{10}
&\left(\begin{array}{cccc}
0&0&0&0\\
0&0&0&0\\
0&0&0&0\\
0&0&0&0
\end{array}\right) & 0 \\
\hline
\A_4^{11}
&\left(\begin{array}{cccc}
0&0&0&0\\
0&0&0&0\\
0&0&0&0\\
0&0&0&0
\end{array}\right) & 0 \\
\hline
\A_4^{12}
&\left(\begin{array}{cccc}
0&0&0&0\\
0&0&0&0\\
a_{31}&0&0&0\\
0&a_{42}&0&0
\end{array}\right) & 2 \\
\hline
\A_4^{13}
&\left(\begin{array}{cccc}
0&0&0&0\\
0&0&0&0\\
0&a_{41}&0&0\\
a_{41}&a_{42}&0&0
\end{array}\right) & 2 \\
\hline
\A_4^{14}
&\left(\begin{array}{cccc}
0&0&0&0\\
0&0&0&0\\
0&a_{32}&0&0\\
0&a_{42}&0&0
\end{array}\right) & 2 \\
\hline
\A_4^{15}
&\left(\begin{array}{cccc}
0&0&0&0\\
0&0&0&0\\
0&a_{32}&0&0\\
0&a_{42}&0&0
\end{array}\right)\; (\alpha \neq -1)
& 2 \\
&\left(\begin{array}{cccc}
a_{11}&a_{12}&a_{13}&a_{14}\\
0&0&0&0\\
-a_{41}&a_{32}&-a_{43}&-a_{44}\\
a_{41}&a_{42}&a_{43}&a_{44}
\end{array}\right)\; (\alpha = -1)
& 9 \\
\hline
\A_4^{16}
&\left(\begin{array}{cccc}
0&0&0&0\\
0&0&0&0\\
0&0&0&0\\
0&0&0&0
\end{array}\right) & 0 \\
\hline
\end{array}
\]

\begin{cor}
\begin{enumerate}
    \item The dimensions of the central derivation of the complex two-dimensional
Zinbiel algebras are one.
\item The dimensions of the central derivation of complex three-dimensional Zinbiel algebras
vary between zero and nine.
\item The dimensions of the central derivation of four-dimensional complex Zinbiel algebras
vary between zero and nine.\cite{a,b,c,d,e,f,g,h,i,j,k,l,m,n,o,p,q,r,s,t}
\end{enumerate}
\end{cor}

\section{Conflicts of Interest}
The authors declare no conflicts of interest.


\begin{thebibliography}{999}

\bibitem{1}Loday, J. L. (1995). Cup-product for Leibniz cohomology and dual Leibniz algebras. Mathematica Scandinavica, 189-196.
\bibitem{2}Loday, J. L., Chapoton, F., Frabetti, A., Goichot, F., \& Loday, J. L. (2001). Dialgebras (pp. 7-66). Springer Berlin Heidelberg.
\bibitem{3}Livernet, M. (1998). Rational homotopy of Leibniz algebras. manuscripta mathematica, 96, 295-315.
\bibitem{4}Chapoton, F. (2002). Un théorème de Cartier–Milnor–Moore–Quillen pour les bigèbres dendriformes et les algèbres braces. Journal of Pure and Applied Algebra, 168(1), 1-18.
\bibitem{5}Ronco, M. (2000). Primitive elements in a free dendriform algebra. Contemporary Mathematics, 267, 245-264.
\bibitem{6}Omirov, B. A. (2002). Classification of two-dimensional complex Zinbiel algebras. Uzbek. Mat. Zh, 2, 55-59.
\bibitem{7}Alvarez, M. A., Júnior, R. F., \& Kaygorodov, I. (2022). The algebraic and geometric classification of Zinbiel algebras. Journal of Pure and Applied Algebra, 226(11), 107106.
\bibitem{8}Adashev, J. Q., Khudoyberdiyev, A. K., \& Omirov, B. A. (2010). Classifications of some classes of Zinbiel algebras. Journal of Generalized Lie Theory and Applications, 4, 1-10.
\bibitem{9}Ni, J. (2014). Centroids of Zinbiel algebras. Communications in Algebra, 42(4), 1844-1853.
\bibitem{10}Almutairi, H., \& AbdGhafur, A. (2018). Derivations of some classes of Zinbiel algebras. International Journal of Pure and Applied Mathematics, 2, 12-13.

\bibitem{a}Sania, A., Imed, B., Mosbahi, B., \& Saber, N. (2023). Cohomology of compatible BiHom-Lie algebras. arXiv preprint arXiv:2303.12906.
\bibitem{b}Zahari, A., Mosbahi, B., \& Basdouri, I. (2023). Classification, Derivations and Centroids of Low-Dimensional Complex BiHom-Trialgebras. arXiv preprint arXiv:2304.06781.
\bibitem{c}Zahari, A., Mosbahi, B., \& Basdouri, I. (2023). Classification, Derivations and Centroids of Low-Dimensional Complex BiHom-Trialgebras. arXiv preprint arXiv:2304.06781.
\bibitem{d}Mosbahi, B., Zahari, A., \& Basdouri, I. (2023). Classification, $\alpha $-Inner Derivations and $\alpha $-Centroids of Finite-Dimensional Complex Hom-Trialgebras. arXiv preprint arXiv:2305.00471.
\bibitem{e}Mosbahi, B., Asif, S., \& Zahari, A. (2023). Classification of tridendriform algebra and related structures. arXiv preprint arXiv:2305.08513.
\bibitem{f}Fiidow, M. A., Zahari, A., \& Mosbahi, B. (2023). Quasi-Centroids and Quasi-Derivations of Low Dimensional Associative Algebras. arXiv preprint arXiv:2306.14331.
\bibitem{g}Asif, S., Wang, Y., Mosbahi, B., \& Basdouri, I. (2023). Cohomology and deformation theory of $\mathcal {O} $-operators on Hom-Lie conformal algebras. arXiv preprint arXiv:2312.04121.
\bibitem{h}Mansuroglu, N., \& Mosbahi, B. (2024). On structures of BiHom-Superdialgebras and their derivations. arXiv preprint arXiv:2404.12098.
\bibitem{i}Mansuroglu, N., \& Mosbahi, B. (2024). Generalized derivations of BiHom-supertrialgebras. arXiv preprint arXiv:2404.12112.
\bibitem{j}Mainellis, E., Mosbahi, B., \& Zahari, A. (2024). Cohomology of BiHom-Associative Trialgebras. arXiv preprint arXiv:2404.15567.
\bibitem{k}Mainellis, E., Mosbahi, B., \& Zahari, A. (2024). Compatible Associative Algebras and Some Invariants. arXiv preprint arXiv:2405.18243.
\bibitem{l}Imed, B., \& Mosbahi, B. (2024). Classification of ($\rho,\tau,\sigma $)-derivations of two-dimensional left-symmetric dialgebras. arXiv preprint arXiv:2411.05716.
\bibitem{m}Imed, B., Lerbet, J., \& Mosbahi, B. (2024). Quasi-Centroids and Quasi-Derivations of low-dimensional Zinbiel algebras. arXiv preprint arXiv:2411.09532.
\bibitem{n}Mosbahi, M., Elgasri, S., Lajnef, M., Mosbahi, B., \& Driss, Z. (2021). Performance enhancement of a twisted Savonius hydrokinetic turbine with an upstream deflector. International Journal of Green Energy, 18(1), 51-65.
\bibitem{o}Mosbahi, M., Lajnef, M., Derbel, M., Mosbahi, B., Aricò, C., Sinagra, M., \& Driss, Z. (2021). Performance improvement of a drag hydrokinetic turbine. Water, 13(3), 273.
\bibitem{p}Mosbahi, M., Derbel, M., Lajnef, M., Mosbahi, B., Driss, Z., Aricò, C., \& Tucciarelli, T. (2021). Performance study of twisted Darrieus hydrokinetic turbine with novel blade design. Journal of Energy Resources Technology, 143(9), 091302.
\bibitem{q}Mosbahi, M., Lajnef, M., Derbel, M., Mosbahi, B., Driss, Z., Aricò, C., \& Tucciarelli, T. (2021). Performance improvement of a Savonius water rotor with novel blade shapes. Ocean Engineering, 237, 109611.
\bibitem{r}Mosbahi, M., Derbel, M., Hannachi, M., Mosbahi, B., Driss, Z., Aricò, C., \& Tucciarelli, T. (2023). Performance study of spiral Darrieus water rotor with V-shaped blades. Proceedings of the Institution of Mechanical Engineers, Part C: Journal of Mechanical Engineering Science, 237(21), 4979-4990.
\bibitem{s}Mosbahi, B., Zahari, A., Basdouri, I. (2023). Classification, $\alpha$-Inner Derivations and $\alpha$-Centroids of Finite-Dimensional Complex Hom-Trialgebras. Pure and Applied Mathematics Journal, 12(5), 86-97. https://doi.org/10.11648/j.pamj.20231205.12
\bibitem{t}ABDOU, A. Z., \& MOSBAHI, B. (2024). CLASSIFICATION OF COMPATIBLE ASSOCIATIVE ALGEBRAS AND SOME INVARIANTS. Available at SSRN 4877916.
\end{thebibliography}
\end{document}